\documentclass[11pt]{amsart}

\usepackage[utf8]{inputenc}
\usepackage{fullpage,url,amssymb,enumitem,colonequals, mathrsfs,comment}
\usepackage{microtype}
\usepackage{hyperref}
\usepackage{xcolor}
\usepackage{tikz, tikz-cd}
\usepackage{lineno}

\usepackage[pdftex]{}

\newcommand{\F}{\mathbb{F}}

\newcommand{\PP}{\mathbb{P}}

\newcommand{\Fbar}{{\overline{\F}}}

\DeclareMathOperator{\Cl}{Cl}

\DeclareMathOperator{\Frob}{Frob}

\DeclareMathOperator{\Pic}{Pic}

\newcommand{\gon}{\operatorname{gon}}

\newtheorem{theorem}{Theorem}[section]
\newtheorem{lemma}[theorem]{Lemma}

\theoremstyle{definition}
\newtheorem{definition}[theorem]{Definition}
\newtheorem{question}[theorem]{Question}

\theoremstyle{remark}
\newtheorem{remark}[theorem]{Remark}

\usepackage{microtype}
\title[Exponents of Jacobians and relative class groups]{Exponents of Jacobians and relative class groups}

\author{Borys Kadets}
\address{Borys Kadets, Einstein Institute of Mathematics\\ Hebrew University of Jerusalem\\
Jerusalem, Israel}
\email{kadets.math@gmail.com}

\author{Daniel Keliher}
\address{Daniel Keliher, Department of Mathematics and Statistics\\Middlebury College\\
Middlebury, VT, USA}
\email{dkeliher@middlebury.edu}

\date{\today}

\begin{document}
\begin{abstract}
   We prove a lower bound for the exponent of the relative class group $\Pic^0 X_1/\phi^* \Pic^0 X_2$ for a covering of curves $\phi\colon X_1 \to X_2$ over a finite field $\F_q$. The results improve on the existing best bounds (due to Stichtenoth) in the case $X_2=\PP^1$, when the relative class group equals the class group of the function field $\F_q(X_1)$, and are completely new for the genuinely relative situation.
\end{abstract}


\maketitle
\section{Introduction}

Suppose $K$ is a global function field over a finite field $\F_q$. The class group, $\Cl(K)$, of $K$ is known to geometers as the group of rational points on the Jacobian of the smooth, proper, integral curve $X$ such that $K=\F_q(X)$. There has been much work on trying to understand the structure of $\Cl(K)$ in both the language of function field arithmetic and the language of curves over finite fields. We adopt the latter perspective from now on.

The easier part of the structure is the size of $\Cl K = \Pic^0 X$, as it is preserved by isogenies and can be described purely in terms of the eigenvalues of the Frobenius. The Weil bounds give the following inequality
\[\left(q^{1/2}-1\right)^{2g} \leqslant \# \Pic^0X \leqslant \left(q^{1/2}+1\right)^{2g},\]
which can be further refined   \cite{lebacque2014number}. A more delicate question is to understand the group structure of $\Pic^0 X$. Consider a curve $X$ of large genus over a fixed finite field, $\F_q$. Then, since $\Pic^0 X$ is a set of rational points on the $g$-dimensional abelian variety $\Pic_X^0$, the Weil bounds do not directly contradict the possibility that $\Pic^0 X$ is $N$-torsion, for $N$ bounded with respect to $g$ (since there are $N^{2g}$ torsion points of order $N$). Thus it is interesting to study the \emph{exponent} of the group $\Pic^0 X$.

The study of class group exponents has an honorable pedigree: the problem of classifying hyperelliptic function fields of class number two appears in Emil Artin's thesis \cite{ArtinThesis1, ArtinThesis2} (which itself was one of the starting points for function field arithmetic \cite{RoquetteBook, oort2016early}). A series of investigations followed, e.g.,~\cite{LeitzelMadanQueen, MaddenQuadratic}, studying function fields with small class number. Finally, with the help of modern computational tools, Bautista-Ancona and Diaz-Vargas \cite{BautistsaAnconaDiazVargas} were able to list all hyperelliptic curves with $2$-torsion class group.

The fact that the exponent of the class group grows with the genus was first established in a series of works of Madan and Madden \cite{MadanMadden1976, MadanMadden1977, MadanMadden1977_degree}. The bound was further improved by the following result of Stichtenoth \cite{Stichtenoth1979}.

\begin{theorem}[\protect{\cite[Theorem 1]{Stichtenoth1979}}]
Suppose $X/\F_q$ is a smooth, proper, integral curve of genus $g$. Then the exponent $e$ of $\Pic^0 X$ satisfies 
    \[e > \frac{g^{1/3}}{4 \log g}.\]
\end{theorem}

The proofs of these results were given in field-theoretic language and rely on choosing a sufficiently nice transcendental element in the field, i.e.,~a covering $f: X \to \PP^1$. Our first main result is an improvement of the bound, proved using a more direct geometric approach. Our bound is stated additionally in terms of the \emph{gonality}, $\gon X$, of $X$: the minimal degree of a rational function $\gamma: X \to \PP^1$. 
\begin{theorem}\label{thm:exponent}
    Suppose $X/\F_q$ is a smooth proper integral curve of gonality $\gon X>1$. Then the exponent $e$ of $\Pic^0 X$ satisfies
    \[e \geqslant \max\left(\frac{\gon X}{2\left\lceil\log_q (2g+1)\right\rceil}, \frac{g}{4(\gon X -1)\left\lceil\log_q(7g+1)\right\rceil}\right).\]
    In particular, regardless of the value of $\gon X$, we have 
    \[e \geqslant \frac{g^{1/2}}{4\lceil\log_q(7g+1)\rceil}.\]
\end{theorem}

The theorem above comes as a consequence of Theorem \ref{thm:mainfull}, which additionally estimates the number of elements of $\Pic^0 X$ of a given order.

The advantage of the geometric approach employed here, beyond  the analytic improvement, is an increased flexibility. By using essentially the same arguments, we are able to obtain results for the relative class group; a case in which no bounds were previously established. 

\begin{theorem}[Theorem \ref{thm:relative-full}]\label{thm:relative}
      Suppose $\phi:X_1 \to X_2$ is a separable covering of curves of respective genera $g_1, g_2$. Suppose $g_1 \geqslant 2$ and $g_2>0$. Consider the exponent $e$ of the relative class group $\Pic^0 X_1/\phi^{*}\Pic^0 X_2$. Then 
\begin{enumerate}
    \item $e \geqslant \dfrac{g_1^{1/2}}{8 \lceil \log_q(19g_1) \rceil (\deg \phi-1)}$;
    \item 
    $e \geqslant \left\lfloor c \dfrac{ \sqrt{g_1}}{g_2} \right\rfloor$, where $c=\min \left(1/16, \dfrac{g_2}{4\log_q(14g_1+1)}\right).$
\end{enumerate}
\end{theorem}

In stating Theorems \ref{thm:exponent} and \ref{thm:relative} we have made an effort to make nice-looking, easily comparable bounds, without sacrificing their asymptotic shape. The proof methods may be also used directly for specific numerical values, improving the estimates; see Remark \ref{rmk:relativefornumerics}.

\begin{remark}
Effectively, Theorem \ref{thm:relative} provides two bounds, one of which is better when $\deg \phi$ is small, and the other is better when $\deg \phi$ is large. These two bounds leave a scant collection of covers for which the exponent is still unbounded from below. What remains are the covers for which $g_1 \approx g_2^2$ and $\deg(\phi) \approx g_2$ (and which are thus almost unramified). In this regime, the exponent bounds become trivial. We do not know if this numerical coincidence reflects a shortcoming of our method, or if low exponent coverings do not in fact exist. We raise a precise version of this problem as Question \ref{mainquestion}.
\end{remark}
\begin{remark}
    Theorem \ref{thm:exponent} has no analogue for general abelian varieties. A simple counterexample is the power $E^n$ of an elliptic curve, whose exponent does not depend on $n$. Even restricting to simple abelian varieties does not resolve the issue, as there are infinitely many simple abelian varieties $A$ over $\F_2$ with $A(\F_2)=0$; see \cite{madan1977abelian}. For more recent work on possible group structures of $A(\F_q)$, see \cite{marseglia2025abelian}.
\end{remark}
\section{Outline and a sketch of the main argument}

We first describe the main idea of the proof of Theorem \ref{thm:exponent}. If $X/\F_q$ is a curve of genus $g\gg 0$, then by the Weil bounds, the set $X(\F_{q^k})$ contains at least two distinct Galois orbits $\tilde{P}, \tilde{Q}$ for $k \approx \log_q(g)$.  If $e$ is the exponent of $\Pic^0 X$, then we must have a linear equivalence of divisors $e\tilde{P} \equiv e\tilde{Q}$. Thus we can produce a map $f: X \to \PP^1$ of degree $e \log_q(g)$ realizing this equivalence. However, curves do not typically have multiple low-degree rational functions. If $f$ and $g$ are two rational functions of degree less than $g^{1/2}$, then by the Castelnuovo--Severi inequality (Theorem \ref{thm:CS}) there exists a nontrivial morphism from $X$ to another curve $Y$ such that both $f$ and $g$ factor through $Y$. Assuming such factorization does not exist implies the bound $e \gtrapprox \sqrt{g}/\log_q(g)$, as claimed in Theorem \ref{thm:exponent}. 

To turn this sketch into a complete argument, we analyze the ramification data attached to such maps $f$ and a fixed gonal map $\gamma$ to contradict the existence of $Y$. We do this in Section \ref{sec:pfofmain}, where we prove  a more general result, Theorem \ref{thm:mainfull}, from which Theorem \ref{thm:exponent} follows.

In Section \ref{sec:relative}, we turn to the relative setting. Given a morphism of curves $\phi: X_1 \to X_2$ we are similarly able to bound, in Theorem \ref{thm:relative-full}, the exponent $e$ of the relative class group $\Pic^0 X_1/ \phi^* \Pic^0 X_2$. Here we have to replace the linear equivalences $e\tilde{P} \equiv e \tilde{Q}$ by more complicated equivalences between low degree divisors. The main geometric observation is that a low degree divisor $D$ on $X_1$ gives rise to a linear equivalence $e(\phi^*\phi_*D)\equiv (e\deg \phi) D$, which again produces a rational function on $X_1$ with a controlled ramification pattern. 

\section{Conventions and Tools}

Throughout we work exclusively in the geometric category of curves over finite fields and their coverings. For a low-tech introduction see \cite{FultonCurves}. For treatments from the function field perspective, see \cite{StichtenothBook} and \cite{RosenBook}.
Throughout, the word \emph{curve} refers to a smooth, projective, geometrically integral, one-dimensional scheme over a finite field $\F_q$. The word \emph{covering} means a finite morphism $\phi\colon X \to Y$ over $\F_q$.  We use $|X|$ to denote the set of closed points of $X$; recall that these are in bijection with Galois orbits in $X(\Fbar_q)$, with the length of an orbit corresponding to the degree of the point. It is convenient to view a divisor $D$ on $X$ as a Galois-invariant sum of points in $X(\Fbar_q)$ (rather than of closed points), and the \emph{underlying set} of $D$ is the support of $D$ viewed as a subset of $X(\Fbar_q)$. Given a point $P \in X(\F_{q^k})$, a degree $k$ divisor \emph{associated} to $P$ is the unique degree $k$ multiple of the Galois orbit of $P$ (so that if $P$ is in fact a rational point, the associated divisor is $kP$).

\subsection{Classical formulas}
We will often use two standard formulas from the geometry of curves.

\begin{theorem}[Castelnuovo--Severi inequality; see \cite{kani1984castelnuovo}]\label{thm:CS}

Let $X,Y_1,Y_2$ be curves over $\mathbb{F}_q$ of genera $g_X,g_1,g_2$ equipped with maps $\pi_1: X \to Y_1$ and $\pi_2: X \to Y_2$ of degrees $n_1, n_2$, respectively. Suppose further there is no morphism $X \to Z$ of degree greater than one through which both $\pi_1$ and $\pi_2$ factor. Then, 
\begin{equation*}
    g_X \leq n_1g_1 + n_2g_2 + (n_1-1)(n_2-1).
\end{equation*}
\end{theorem}

\begin{theorem}[Riemann--Hurwitz inequality]\label{thm:RH}
    Suppose $\phi: X \to Y$ is a separable covering of curves. For a point $P \in |X|$ let $e(P)$ denote the ramification index of $\phi$ at $P$. Then the genera $g_X, g_Y$ of $X$ and $Y$ satisfy 
    \[2g_X-2 \geqslant \deg \phi(2g_Y-2)+\sum_{P \in |X|}\deg P(e(P)-1).\]
    In particular, the number of branch points in $X(\Fbar_q)$ is bounded from above by $(2g_X-2)-\deg \phi (2g_Y-2).$
\end{theorem}

\section{Exponents of Jacobians}\label{sec:pfofmain}

We first need to introduce a term for divisors which interact in as simple a way as possible with a rational function $f$.

\begin{definition}
    Given a curve $X$, a covering $f: X \to Y$, and an effective divisor $D$ on $X$, we say that $D$ is \emph{non-fibral for $f$} if the sets $f(D)$ and $D$ are equal in size, and the ramification index of $f$ at every point of $D$ is equal to the inseparable degree of $f$.
\end{definition}
Observe that if $f=i \circ j$, and $D$ is non-fibral for $f$, then $D$ is non-fibral for $j$ as well. Heuristically, most divisors on a curve should be non-fibral relative to a fixed $f$. The following lemma shows that there are many low degree points on a curve which, when viewed as divisors, are non-fibral relative to $f$.
\begin{lemma}\label{lem:non-fibralcount}
    If $X/\F_q$ is a curve of genus $g>0$, $f: X \to \PP^1$ is a rational function, and $k$ is a prime number, then there are at least 
    \[\left\lceil\frac{q^k}{k} - 2g\frac{q^{k/2}+1}{k} - \deg f \frac{q+3}{k}\right\rceil\] 
    degree $k$ points $P \in |X|$ which, when viewed as degree $k$ divisors, are non-fibral under $f$. 
\end{lemma}

\begin{proof}
    Since $g$ is greater than $0$, $f$ is not a power of the Frobenius. Thus $f$ factors as follows
    \begin{center}
        \begin{tikzcd}
        X \arrow[r, "\mathrm{Frob}_{p^n}"] \arrow[bend right=30]{rr}{f} & X^{(p^n)} \arrow[r, "f'"] & \mathbb{P}^1,
        \end{tikzcd}
    \end{center}
    where $f':X^{(p^n)} \to \PP^1$ is a separable, nonconstant map from the Frobenius twist. Since $\Frob_{p^n}$ is a bijection between degree $k$ points on $X$ and $X^{(p^n)}$, it is enough to prove the result for the separable morphism $f'$.

    By the Weil conjectures, there are at least $q^k-2gq^{k/2}+1$ points in $X(\F_{q^k})$. At most $(\deg f') (q+1)$ among these points are sent to $\PP^1(\F_q)$ by $f'$. Moreover, since $f'$ is separable, by the Riemann--Hurwitz formula (Theorem \ref{thm:RH}), at most $2\deg f' + 2g -2$ points on $X$ are ramified. Thus we can choose $N$ points among $X(\F_{q^k})$, whose image under $f'$ does not belong to $\PP^1(\F_q)$, and which do not intersect the ramification locus of $f'$ for \[N=(q^k-2gq^{k/2}+1) - \deg f' (q+1) - (2\deg f' + 2g - 2).\]
    Note that, since $k$ is prime, any point in $X(\F_{q^k})\setminus X(\F_q)$, in particular those among the $N$ chosen, come from a degree $k$ point in $|X|$. Passing to Galois orbits, we have at least $\lceil N/k \rceil$ non-fibral degree $k$ points in $|X|$.
\end{proof}

We are now ready to prove the main result in the non-relative setting. For the applications to the relative setting, we need more than just a bound on the exponent, we also need an estimate on the \emph{number} of points in $\Pic^0 X$ of a given order $m$. We give both in the following theorem, from which Theorem \ref{thm:exponent} is an immediate consequence.
\begin{theorem}\label{thm:mainfull}
    Suppose $X/\F_q$ is a curve of gonality $\gon X>1$. Then the exponent $e$ of $\Pic^0 X$ satisfies
    \[e \geqslant \max\left(\frac{\gon X}{2\left\lceil\log_q (2g+1)\right\rceil}, \frac{g}{4(\gon X -1)\left\lceil\log_q(7g+1)\right\rceil}\right).\]
Moreover, given an integer $m$, letting $s$ denote $\lceil\gon X / m\rceil - 1$, and letting $k$ be the largest prime less than $g/(m(\gon X -1))$, one can choose at least 

\[N(m)=\max \left(q^{s/2}-3g, q^{k/2}-7g \right) \geqslant q^{\sqrt{g}/4m} - 7g\]

points $Q_1, \dots Q_N\in \Pic^0 X$, such that the order of the $Q_i$ and of the pairwise differences $Q_i-Q_j$ is at least $m$.
\end{theorem}
\begin{proof} We prove both inequalities on the exponent by building a pair $(P,Q)$ of low degree points on $X$ using Lemma \ref{lem:non-fibralcount} and showing, for low values of $e$, that the linear equivalence $eP \equiv eQ$ is not possible. For the second claim, we use the same argument to show that it is possible to build $N+1\geqslant N(m)+1$ points $P_1, \dots, P_{N+1}$ such that all pairwise differences $P_i-P_j$ have order at least $m$; taking $Q_i=P_i-P_{N+1}$ for $i=1,\dots, N$ proves the claim.
 
\noindent \textbf{Step 1: The first term of each maximum.}

We first prove the inequality \[e \geqslant \frac{\gon X}{2 \left\lceil\log_q(2g+1)\right\rceil}.\]

By the Weil bounds, there are at least\footnote{the inequality can be seen by collecting $q^{s}$-term on the left and $q^{s/2}$-terms on the right} two non-Galois-conjugate points $P,Q \in X(\F_{q^d})$ for  $d=\left\lceil 2\log_q(2g+1)\right\rceil$. Let $\tilde{P}, \tilde{Q}$ be degree $d$ divisors on $X$ associated to $P, Q$. If \[e < \dfrac{\gon X}{2\left\lceil \log_q(2g+1)\right\rceil}\] holds, then the linear equivalence $e\tilde{P} \equiv e\tilde{Q}$ is an equivalence of distinct effective divisors of degree less than $\gon X$, and so the rational function realizing this equivalence would have degree less than the gonality, a contradiction. Similarly, by the Weil bounds there are at least 
\[\frac{1}{s}\left(q^s-2gq^{s/2}+1\right)\geqslant q^{s/2} - 2g\geqslant q^{s/2}-3g+1 \]
effective divisors on $X$ of degree $s=\lceil \gon X/m\rceil-1$, obtained from Galois orbits of points in $X(\F_{q^s})$. Denote these divisors by $D_1, \dots D_{N+1}$, with $N\geqslant q^{s/2}-3g$.  Since no two distinct effective divisors of degree $ms$ are linearly equivalent, taking $Q_i=D_i-D_{N+1}$ for $1\leqslant i \leqslant N$, proves the second claim of the theorem (with the first term of the maximum).  

\noindent\textbf{Algebra Intermission:  The second bound on $e$. }

What is left to prove is the existence of at least  
$ q^{k/2} -7g  $ 
points in $Q_1, \dots, Q_N\in \Pic^0X$ of order at least $m$, whose pairwise differences have order at least $m$, as claimed in the theorem. Combined with the above discussion, this would complete the proof that there are at least 
\[N(m) = \max \left(q^{s/2} - 3g, q^{k/2}-7g   \right)\]
elements of $\Pic^0 X$ with the desired properties. This also implies the full bound on the exponent, since  for \[m = \left\lceil \frac{g}{4(\gon X - 1)\left\lceil\log_q(7g+1)\right\rceil} \right \rceil, \]  Bertrand's postulate applied to bound $k$ gives $N(m)\geqslant 1$, thus proving the inequality \[e \geqslant  \frac{g}{4(\gon X - 1)\left\lceil\log_q(7g+1)\right\rceil}. \]

\noindent\textbf{Step 2:   The second bound, $N(m) \geqslant  q^{k/2}-7g $.}

Suppose $\gamma: X \to \PP^1$ is a function of degree $\gon X$. Note that when $\gon X > g$, the inequality becomes trivial, so we assume in addition  $\gon X \leqslant g$; in fact such an inequality holds for all but finitely many curves by the work of Faber--Grantham--Howe \cite{faber2024maximumgonalitycurvefinite}. By Lemma \ref{lem:non-fibralcount} there are at least $N+1$ degree $k$ points which are non-fibral with respect to $\gamma$ whenever $k$ is a prime such that

    \[q^k - 2g(q^{k/2} + 1) - \deg \gamma (q+3) > Nk. \]
   Since $q \geq 2$ and $k \geq 2$, the previous displayed inequality holds whenever 

\[ \frac{1}{k}(q^k - 3gq^{k/2} - \deg \gamma (q+3)) > N.\]

  Further, since $\deg \gamma \leqslant \gon(X)$, $\gon(X) \leqslant g$, and $q+3 < 3q^{k/2}$, the previous displayed inequality holds whenever

\[ \frac{1}{k}(q^k - 3gq^{k/2} - 3gq^{k/2}) \geqslant N.\]

Equivalently, 

\[ N \leqslant \frac{1}{k}(q^k - 6gq^{k/2}).\]

Let $k$ be the largest prime less than $g/(m (\gon X - 1))$.
With this choice of $k$, denote by $P_1,...,P_{N}, P_{N+1}$ degree $k$ points that are non-fibral with respect to $\gamma$. We now prove that the pairwise differences $P_i-P_j$ have order at least  \[m=\left \lceil \frac{g}{k (\gon X - 1)} \right \rceil;\]
since we can take \[N=\left\lfloor \frac{1}{k}\left(q^{k}-6gq^{k/2}\right)\right\rfloor\geqslant q^{k/2}-7g\]
this will prove the inequality  \[N(m)\geqslant q^{k/2}-7g\] for $k$ the largest prime less than \[\frac{g}{m (\gon(X)-1)}.\]

Suppose $f: X \to \PP^1$ is a function of degree $ek$ associated to the linear equivalence $eP_i \equiv eP_j$. If $g>(\gon X - 1)(ek-1)$, then the map $(f, \gamma): X \to \PP^1 \times \PP^1$ cannot be birational to its image, and so there is a covering $\phi: X \to Z$ through which both $\gamma$ and $f$ factor. Moreover, the genus of $Z$ is less than $(\gon X - 1)(ek-1)$, and so $\phi$ is not purely inseparable. We assumed that both $P_i, P_j$ are non-fibral with respect to $\gamma$, and so the same is true under $\phi$. Consider the set $f^{-1}(f(P_i))$. On one hand, the construction of $f$ ensures that this set consists of the single point $P_i$. On the other hand, $f^{-1}(f(P_i))$ contains $\phi^{-1}(\phi(P_i))$, which has points not from $P_i$, since $P_i$ is non-fibral under $\phi$; this is a contradiction. Thus $g<(\gon X -1)(ek-1)$, which gives the desired inequality.    \qedhere
\end{proof}

\section{Exponents of relative class groups}\label{sec:relative}

Here we restate and prove Theorem \ref{thm:relative}, which addresses the case of the relative class group $\Pic^0 X_1/\phi^{*}\Pic^0 X_2$ of a separable covering $\phi: X_1 \to X_2$. 
\begin{theorem}\label{thm:relative-full}
    Suppose $\phi:X_1 \to X_2$ is a separable covering of curves of respective genera $g_1, g_2$. Suppose $g_1 \geqslant 2$ and $g_2>0$. Consider the exponent $e$ of the relative class group $\Pic^0 X_1/\phi^{*}\Pic^0 X_2$. Then the following hold
\begin{enumerate}
    \item\label{case:highdegr} $e \geqslant \lfloor c \sqrt{g_1}/g_2 \rfloor$, where $c=\min \left(1/16, \dfrac{g_2}{4\log_q(14g_1+1)}\right);$
    \item[]
    \item $e \geqslant \dfrac{g_1^{1/2}}{8 \lceil\log_q(19g_1)\rceil (\deg \phi -1)}.$\label{case:lowdegr}
\end{enumerate}
\end{theorem}

\begin{proof}
We first show part \eqref{case:highdegr} of the claim. By Theorem \ref{thm:mainfull}, there are at least $N(m)$ points on $\Pic^0 X_1$ such that their order, and the order of their pairwise differences is at least $m$. If $N(m)$ is larger than $\#\Pic^0 X_2$, then the exponent $e$ of the quotient $\Pic^0X_1/\phi^*\Pic^0 X_2$ is at least $m$. To finish the proof of the first claim, we just need to compare the Weil bound $\# \Pic^0 X_2 < (\sqrt{q}+1)^{2g_2}$ with the formula for $N(m)$. In other words, whenever
\begin{equation*}\label{eq:mfromNm}
q^{\sqrt{g_1}/4m}-7g_1 > (\sqrt{q}+1)^{2g_2}+1
\tag{$\diamond$}
\end{equation*}
 holds, the exponent $e$ is at least $m$. Comparing the main terms, we deduce that, up to a constant factor, this holds when $m \approx \sqrt{g_1}/(2g_2)$. To get a cleaner exact statement (with non-optimal multiplicative constant), we set $c=\min\left(1/16, \frac{g_2}{4\log_q(14g_1+1)}\right)$, and let $m=\lfloor c\sqrt{g_1}/g_2 \rfloor$. With this notation, we observe that, since \[m<c\sqrt{g_1}/g_2 < \frac{\sqrt{g_1}}{4\log_q(14g_1+1)},\] we get
 \begin{equation*}
     7g_1<\frac{1}{2}q^{\sqrt{g_1}/4m}.
\end{equation*}
Using the previous line, we can bound left hand side of \eqref{eq:mfromNm} from below like so:

\[q^{\sqrt{g_1}/4m}-7g_1>q^{\sqrt{g_1}/4m - 1}.\]

At the same time, the right hand side of \eqref{eq:mfromNm} is bounded above by $q^{2g_2+1}$. Thus the inequality \eqref{eq:mfromNm} is satisfied whenever \[\frac{\sqrt{g_1}}{4m} \geqslant 2+2g_2.\]
Using $m\leqslant  c\sqrt{g_1}/g_2$ gives 
\[g_2/4c\geqslant 2+2g_2,\] which is trivially satisfied since $c \leqslant 1/16$. This completes the proof of case \eqref{case:highdegr}.

We now prove claim \eqref{case:lowdegr}. Let $\gamma: X_1 \to \PP^1$ denote a gonal map for $X_1$. The idea is to proceed similarly to the proof of Theorem \ref{thm:mainfull}: we will consider divisors of the form $D=\tilde{P}-\tilde{Q}$ for low degree points $P, Q$ on $X_1$, and will prove that $eD$ is not obtained by pulling back a divisor class from $X_2$. The key property of divisors $D$ pulled back from $X_2$ is $\phi^*\phi_* D \equiv \deg \phi \ D$; this is a linear equivalence between low degree divisors, which for suitable $D$ cannot occur.

To begin, we need to produce two low degree points $P,Q$ on $X_1$ that behave generically with respect to $\phi$ and $\gamma$, in a similar vein to Lemma \ref{lem:non-fibralcount}. We want to find the points meeting the three conditions below. 
    \begin{enumerate}[label={(\alph*)}]
    \item The Galois orbits $\tilde{P}, \tilde{Q} \subset X_1$ of $P,Q$ are non-fibral with respect to $\gamma$ and $\phi$.\label{thm2:cond1}
    \vspace{1mm}
    \item The Galois orbits of $\phi(P), \phi(Q) \in X_2$ are distinct.  \label{thm2:cond2}
    \vspace{1mm}
    \item Let $\tilde{Q}'$ denote the effective divisor $\phi^*(\phi_*(\tilde{Q}))-\tilde{Q}$. Then if $\tilde{Q}'$ happens to contain $k$ points (with multiplicities $\deg \phi -1$), then $\gamma(\tilde{Q}') \neq \gamma(\tilde{P})$.\footnote{This specific behavior of multiplicities of points in $\tilde{Q}'$ is probably very rare except for the case $\deg \phi =2 $.} \label{thm2:cond3}
    \end{enumerate}
Let $k$ be prime. We proceed as in the proof of Theorem \ref{thm:mainfull}. Combining the Weil bounds with the Riemann--Hurwitz formula, the set $M_{\phi, \gamma}$ of points in $X_1(\F_{q^k})$ whose Galois orbits are non-fibral with respect to $\phi$ and $\gamma$ has size at least
\begin{multline*}
    \#M_{\phi, \gamma} \geqslant\#X_1(\F_{q^k}) - \deg \phi\#X_2(\F_q) - \#\mathrm{Ram} (\phi) - \deg \gamma\#\PP^1(\F_q) - \#\mathrm{Ram}(\gamma) 
    \geqslant \\
 (q^k - 2g_1q^{k/2}  + 1)  - (q+ 2g_2 q^{1/2} + 1) \deg \phi  - (2g_1 - 2)  - (q+ 1) \gon X_1 -(2\gon X_1 +2g_1 - 2);
\end{multline*}

here $\mathrm{Ram}(f)$ denotes the locus where the ramification index of a covering $f:X \to Y$ is larger than its inseparable degree.

Once $P$ is chosen, conditions \ref{thm2:cond2} and \ref{thm2:cond3} remove at most $k \deg \phi + k \gon X_1$ points from $M_{\phi, \gamma}$ as candidates for $Q$. Thus we can produce two points $P,Q$ satisfying assumptions \ref{thm2:cond1}, \ref{thm2:cond2}, \ref{thm2:cond3} whenever
 \begin{multline*} (q^k - 2g_1q^{k/2}  + 1) - (2\gon X_1 +2g_1 - 2) - (2g_1 - 2)  \\ - (q+ 1) \gon X_1 - (q+ 2g_2 q^{1/2} + 1) \deg \phi \geqslant k(\gon X_1 + \deg \phi) + 1.
 \end{multline*}

Since $\gon X_1 \leqslant g_1 + 1$, the above holds whenever 
\[(q^k - 2g_1q^{k/2} ) - (6g_1 - 2) - (q+ 1)(g_1 + 1) - (q+ 2g_2 q^{1/2} + 1) \deg \phi \geqslant k(\gon X_1 + \deg \phi).\]
Collecting $k$-dependent terms on the left, we have
\[q^k - 2g_1q^{k/2} - k(\gon X_1 + \deg \phi)  \geqslant (6g_1 - 2) + (q+ 1)(g_1 + 1) + (q+ 2g_2 q^{1/2} + 1) \deg \phi  . \]
Since part \eqref{case:lowdegr} holds trivially for $\deg \phi >1/4 g_1^{1/2}$, we have $\deg \phi + \gon X_1 \leqslant 2g_1$ and so we simplify to

\[q^k - 2g_1q^{k/2} - 2g_1q^{k/2}  \geqslant (6g_1 - 2) + (q+ 1)(g_1 + 1) + (q+ 2g_2 q^{1/2} + 1) \deg \phi  . \]
Since, by Riemann--Hurwitz, $(2g_2-2)\deg \phi \leqslant 2g_1-2$ we have

\[q^k - 4g_1q^{k/2}  \geqslant 6g_1q^{k/2} + 4g_1q^{k/2} + 2g_1q^{k/2} + (q+ 2q^{1/2} + 1) \deg \phi  . \]

Since the statement of part \eqref{case:lowdegr} becomes trivial when  $\deg \phi > \sqrt{g_1}$, we can assume that $\deg \phi \leqslant \sqrt{g_1}$, and so it is enough to verify

\[q^{k}-16g_1q^{k/2} \geqslant 3q^{k/2}g_1.\]
Whence,
\[k \geqslant 2 \log_q(19 g_1).\] 
 
Let $2\lceil\log_q(19g_1)\rceil \leqslant k  \leqslant 4\lceil\log_q(19g_1)\rceil$ be a prime number. Now we may choose $P, Q \in X(\F_{q^k})$ satisfying \ref{thm2:cond1}, \ref{thm2:cond2} and \ref{thm2:cond3}. Let $\tilde{P}, \tilde{Q}$ denote the Galois orbits of $P, Q$ viewed as degree $k$ divisors.  Let $\tilde{P}'$ denote the effective divisor $\phi^*\phi_* \tilde P - \tilde P$, and similarly for $\tilde{Q}'$. Since $\phi (\tilde{P}) \neq \phi(\tilde{Q})$, we have in addition that $\tilde{P}' \neq \tilde{Q}$. We claim that if $e$ is small enough, the divisor class $e(\tilde{P}-\tilde{Q}) \in \Pic^0 X_1$ does not belong to $\phi^* \Pic^0 X_2$. Indeed, otherwise we would have a linear equivalence \[e \phi^*\phi_*(\tilde{P}-\tilde{Q}) \equiv  e (\deg \phi) (\tilde{P}-\tilde{Q}).\]
Rearranging terms this reads
\begin{equation}\label{linear-equiv-2}
e\tilde{P}'+e(\deg \phi - 1)\tilde{Q}\equiv e\tilde{Q}'+e(\deg \phi - 1)\tilde{P}.\tag{$\star$}
\end{equation}

This is an equivalence of effective divisors of degrees $2ek(\deg \phi-1) $. Assume that $e$ is small enough so that \[2ek(\deg \phi -1)  < \max \left(\gon X_1, g_1/(\gon X_1 - 1)\right).\]

We will arrive at a contradiction similarly to the proof of Theorem \ref{thm:mainfull}. Observe that, since $\phi_*\tilde{P} \neq \phi_* \tilde{Q}$, the linear equivalence \eqref{linear-equiv-2} has no base locus. In other words, the two sides of \eqref{linear-equiv-2} are fibers of a rational function $f$. Clearly,  $\deg f = 2ek(\deg \phi-1)$ is at least as large as $\gon X_1$, and so we can assume $2ek(\deg \phi -1)  < g_1/(\gon X_1 -1)$. The last inequality implies  $(\deg f -1)(\deg \gamma -1) < g_1$, and so with the Castelnuovo--Severi inequality (Theorem \ref{thm:CS}) applied to the pair of functions $f, \gamma$, we can find a map $\psi: X_1 \to Z$ through which both $f$ and $\gamma$ factor. Observe that since $\psi$ is a factor of $\gamma$, every Galois conjugate of $P$ is non-fibral, and in particular unramified, with respect to $\psi$. On the other hand, since $D=e\tilde{Q}'+e(\deg \phi -1 )\tilde{P}$ is a fiber of $f$, and $f$ factors through $\psi$, $D$ has to contain the divisor $E=e(\deg \phi-1)\psi^{-1}(\psi(\tilde{P}))$. But $\deg E> \deg D$, unless $\deg \psi =2$. Thus the only possibility is $\deg \psi = 2$, $\tilde{Q}'$ consists of $k$ points with multiplicity $\deg \phi -1$ each, and $\psi(\tilde{Q}')=\psi(\tilde{P})$. Since $\gamma$ factors through $\psi$ this does not occur by condition \ref{thm2:cond3}. This gives the desired contradiction.

Thus \begin{equation}\label{eq:relativemax}e\geqslant \frac{\max \left(\gon X_1, g_1/(\gon X_1 - 1)\right)}{2k( \deg \phi-1)} \geqslant \frac{g_1^{1/2}}{8(\deg \phi -1 )\lceil \log_q(19g_1) \rceil}.\qedhere
\end{equation}
\end{proof}

\begin{remark}\label{rmk:relativefornumerics}
    Using \eqref{eq:relativemax} to bound $e$ numerically may be advantageous relative to using the bound on $e$ from the statement of Theorem \ref{thm:relative-full}. More generally, throughout we have relied on rather rough inequalities (such as $x<q^x$), which do not harm the asymptotic shape of the result, but are best avoided in a computational setting. In the same vein, the strategy of the proof can also be applied to divisors of composite degree, possibly shaving off a factor of $2$ that comes from Bertrand's postulate. 
\end{remark}

Theorem \ref{thm:relative-full} motivates the following question.
\begin{question}\label{mainquestion}
    Given an integer $N$, is the set of separable coverings $\phi: X_1 \to X_2$ of curves of genus larger than $2$ over a finite field, for which the relative class group has exponent $N$, finite?
\end{question}

For $N=1$, this is answered affirmatively in \cite{Kedlaya1, kedlaya2, kedlaya3}, but even the case $N=2$ remains open.  For more general bounds, as well as a solution to the relative class number $2$ problem see \cite{PinerosEtAl}.

Theorem \ref{thm:relative-full} does not answer Question \ref{mainquestion}, though it restricts the possibilities for the triples $(g_1, g_2, \deg \phi)$; in particular, we see that low relative class number implies mild ramification of the covering. According to our inequalities, it is natural to speculate that the most interesting case to consider is that of an unramified cover of degree $g_2-1$.

\section*{Acknowledgments}
We thank David Kazhdan for asking a question that lead us to discover Theorem  \ref{thm:exponent}. We also thank Santiago Arango-Pi{\~n}eros for his comments on an earlier draft of this paper and two anonymous referees for their suggestions. 

\bibliographystyle{alpha}
\bibliography{bibliography}
\end{document}